\documentclass[12pt]{amsart}
\usepackage{ifpdf}

\ifpdf
  \usepackage[pdftex]{epsfig}
  \usepackage[pdftex,colorlinks=true,%
              pdftitle={Reduction theory of point clusters},%
              pdfauthor={Michael Stoll}]{hyperref}
  
\else
  \usepackage{epsfig}
  
\fi

\usepackage{amssymb}  
\usepackage{latexsym} 
\usepackage{comment}

\usepackage[all]{xy}

\newcommand{\Z}{{\mathbb Z}}
\newcommand{\Q}{{\mathbb Q}}
\newcommand{\R}{{\mathbb R}}

\newcommand{\C}{{\mathbb C}}
\newcommand{\BP}{{\mathbb P}}

\newcommand{\CZ}{{\mathcal Z}}

\newcommand{\HH}{{\mathcal H}}
\newcommand{\To}{\longrightarrow}

\newcommand{\SL}{\operatorname{SL}}
\newcommand{\PGL}{\operatorname{PGL}}
\newcommand{\PSL}{\operatorname{PSL}}
\newcommand{\U}{\operatorname{U}}

\newcommand{\half}{\tfrac{1}{2}}

\newcommand{\stab}{{\text{st}}}
\newcommand{\sstab}{{\text{sst}}}

\newcommand{\xx}{{\mathbf x}}

\newenvironment{Proof}{\par\noindent{\sc Proof:}}%
                      {\hspace*{\fill}\nobreak$\Box$\par}

   {\begin{list}{}{\settowidth{\labelwidth}{#1}%
                   \setlength{\leftmargin}{\labelwidth}%
                   \addtolength{\leftmargin}{\labelsep}}}%
   {\end{list}}

\newtheorem{Theorem}{Theorem}
\newtheorem{Lemma}[Theorem]{Lemma}

\newtheorem{Corollary}[Theorem]{Corollary}

\theoremstyle{definition}
\newtheorem{Definition}[Theorem]{Definition}
\newtheorem{Example}[Theorem]{Example}

\newtheorem{Remark}[Theorem]{Remark}

\numberwithin{equation}{section}

\begin{document}

\title[Reduction theory of point clusters]%
      {Reduction theory of point clusters \\ in projective space}

\author{Michael Stoll}
\address{Mathematisches Institut,
         Universit\"at Bayreuth,
         95440 Bayreuth, Germany.}
\email{Michael.Stoll@uni-bayreuth.de}

\date{31 August, 2009}

\maketitle



\section{Introduction}

In this paper, we generalise the results of~\cite{StollCremona}
on the reduction theory of binary forms, which describe positive zero-cycles
in~$\BP^1$, to positive zero-cycles (or point clusters) in projective spaces
of arbitrary dimension. This should have applications to more general
projective varieties in~$\BP^n$, by associating a suitable positive zero-cycle
to them in an $\PGL(n+1)$-invariant way. We discuss this in the case of
(smooth) plane curves.

The basic problem motivating this work is as follows. Consider projective
varieties over~$\Q$ in some~$\BP^n\!$, with fixed discrete invariants. On 
this set, there is an action of $\SL(n+1, \Z)$ by linear
substitution of the coordinates. We would like to be able to select
a specific representative of each orbit, which we will call {\em reduced},
in a way that is as canonical as possible. Hopefully, this representative
will then also allow a description as the zero set of polynomials with
fairly small integer coefficients.

Recall the main ingredients of the approach taken in~\cite{StollCremona}.
The key role is played by a map~$z$ from binary forms of degree~$d$ 
into the symmetric space of~$\SL(2, \R)$ (which is the hyperbolic plane~$\HH$
in this case) that is equivariant with respect to the action of~$\SL(2, \Z)$.
We then define a form~$F$ to be {\em reduced} if $z(F)$ is in the standard
fundamental domain for~$\SL(2, \Z)$ in~$\HH$. In order to make the map~$z$
as canonical as possible, we use a larger group than~$\SL(2, \Z)$, namely
$\SL(2, \C)$; we then look for a map $z$ from binary forms with complex
coefficients into the symmetric space~$\HH_\C$ for~$\SL(2, \C)$ that is
$\SL(2, \C)$-equivariant and commutes with complex conjugation. This map
restricted to real forms will have image contained in~$\HH$ and satisfy
our initial requirement.

Now there are in general many possible such maps~$z$ (for exceptions,
see below). We therefore need to pick one of them. In~\cite{StollCremona}
this is achieved by a geometric property: we define a function on~$\HH_\C$
that measures how far
a point is from the roots of~$F$ (up to an arbitrary additive constant);
the covariant $z(F)$ is then the unique point in~$\HH_\C$ minimising
this distance.

In our more general situation, we work with the space $\HH_{n,\R}$ of
positive definite quadratic forms in $n+1$ variables,
modulo scaling, and the space $\HH_{n,\C}$ of positive definite Hermitian
forms in $n+1$ variables, modulo scaling (by positive real factors).
There is a natural action of complex conjugation on~$\HH_{n,\C}$; the
subset fixed by it can be identified with~$\HH_{n,\R}$.

We use the formula for the distance function mentioned above to obtain a similar
function on~$\HH_{n,\C}$, depending on a collection of points
in~$\BP^n(\C)$. 
Under a suitable condition on the point cluster or zero-cycle~$Z$,
this distance function has a unique critical point, which 
provides a global minimum. We assign this point to~$Z$ as its
covariant~$z(Z)$, thus solving our problem.


\section{Basics}

In all of the paper, we fix $n \ge 0$.

We consider the group $G = \SL(n+1, \C)$ and its natural action on
forms (homogeneous polynomials) in $n+1$ variables $X_0, \dots, X_n$
by linear substitutions; this action will be on the right:
\[ F(X_0, X_1, \dots, X_n) \cdot (a_{ij})_{0 \le i,j \le n}
    = F\Bigl(\sum_{j=0}^n a_{0j} X_j, \dots, \sum_{j=0}^n a_{nj} X_j\Bigr) \,.
\]
The same action is used for Hermitian forms in $X_0, \dots, X_n$.
A Hermitian form can be considered as a bihomogeneous polynomial of
bidegree~$(1,1)$ in two sets of variables $X_0, \dots, X_n$ and
$\bar{X}_0, \dots, \bar{X}_n$, where the action on the second set
is through the complex conjugate of the matrix. The form~$Q$ is Hermitian
if $Q(\bar{X}; X) = \bar{Q}(X; \bar{X})$, where $\bar{Q}$ denotes
the form obtained from~$Q$ by replacing the coefficients with their
complex conjugates. Hermitian forms can also
be identified with Hermitian matrices, i.e., matrices $A$ such that
$A^\top = \bar{A}$, where $A$ corresponds to~$Q$ if $Q(\xx) = \bar{\xx} A \xx^\top$;
then the action of~$G$ is given by $A \cdot \gamma = \bar{\gamma}^\top A \gamma$.

The group $G$ also acts on coordinates $(\xi_0, \dots, \xi_n)$ on the
right via the contragredient representation,
\[ (\xi_0, \dots, \xi_n) \cdot \gamma = (\xi_0, \dots, \xi_n) \gamma^{-\top} \,. \]
These actions are compatible in the sense that
\[ (Q \cdot \gamma)(\xx \cdot \gamma) = Q(\xx) \]
for Hermitian forms~$Q$ and coordinate vectors~$\xx$.


\section{Point Clusters}

The actions described above induce actions of $\PSL(n+1, \C) = \PGL(n+1, \C)$ on
projective schemes over~$\C$ and points in projective space~$\BP^n(\C)$.
The first specialises and the second generalises to an action on
positive zero-cycles. 

\begin{Definition}
  A {\em positive zero-cycle} or {\em point cluster}
  is a formal sum $Z = \sum_{j=1}^m P_j$ of points $P_j \in \BP^n$.
  The number~$m$ of points is the {\em degree} of~$Z$, written $\deg Z$.
  If $L \subset \BP^n$ is a linear subspace, we let $Z|_L$ be the sum
  of those points in~$Z$ that lie in~$L$.
\end{Definition}

\begin{Definition}
  Let $Z$ be a point cluster in~$\BP^n$.
  \begin{enumerate}
    \item $Z$ is {\em split} if there are two disjoint and nonempty linear subspaces
          $L_1, L_2$ of~$\BP^n$ such that
          $Z = Z|_{L_1} + Z|_{L_2}$. Otherwise, $Z$ is {\em non-split}.
    \item $Z$ is {\em semi-stable} if for every linear subspace
          $L \subset \BP^n$, we have 
          \[ (n+1) \deg Z|_L \le (\dim L + 1)  \deg Z \,. \]
    \item $Z$ is {\em stable} if for every linear subspace
          $\emptyset \neq L \subsetneq \BP^n$, we have
          \[ (n+1) \deg Z|_L < (\dim L + 1) \deg Z \,. \]
  \end{enumerate}
\end{Definition}

\begin{Remark}
  Note that a split point cluster cannot be stable.
  
  If we identify the cluster $Z = \sum_{j=1}^m P_j$, where
  $P_j = (a_{j0} : a_{j1} : \ldots : a_{jn})$, with the form
  $F(Z) = \prod_{j=1}^m (a_{j0} x_0 + a_{j1} x_1 + \ldots + a_{jn} x_n)$
  (up to scaling), then $Z$ is \hbox{(semi-)}stable if and only if
  $F(Z)$ is (semi-)stable in the sense of Geometric Invariant Theory,
  see~\cite{Mumford}.
  
  If $n = 1$, then the notions of stable and semi-stable defined here
  coincide with those defined in~\cite{StollCremona} (in Def.~4.1 and
  before Prop.~5.2) for binary forms.
\end{Remark}

\begin{Definition}
  Let $\CZ_m$ denote the set of point clusters of degree~$m$ in~$\BP^n(\C)$, 
  $\CZ_m^\sstab$ the subset of semi-stable and $\CZ_m^\stab$ the subset
  of stable point clusters. We denote by~$\CZ_m(\R)$ etc.\ the subset of
  point clusters fixed by complex conjugation, which acts via
  $\sum_j P_j \mapsto \sum_j \bar{P}_j$.
\end{Definition}

For notational convenience, we define for a point cluster~$Z$ and $-1 \le k \le n$
\[ \varphi_Z(k) = \max\{\deg Z|_L : L \subset \BP^n \text{\ a $k$-dimensional
                                       linear subspace}\} \,.
\]
Then $Z$ is semi-stable if and only if $\varphi_Z(k) \le \frac{k+1}{n+1} \deg Z$
and stable if and only if the inequality is strict for $0 \le k < n$.

We let $\langle P, P' \rangle = \bar{P} (P')^\top$ denote the standard
Hermitian inner product on row vectors and $\|P\|^2 = \langle P, P \rangle$
the corresponding norm. The next lemma is the basis for most of what
follows.

\begin{Lemma} \label{Lemma:LowerBound}
  Let $Z \in \CZ_m$. Fix row vectors $P_j$, $j \in \{1,\ldots,m\}$, representing
  the points in~$Z$, such that $\|P_j\|^2 = 1$.Then there is a constant $c > 0$
  such that for every positive definite Hermitian matrix~$Q$ with eigenvalues
  $0 < \lambda_0 \le \lambda_1 \le \ldots \le \lambda_n$, we have
  \[ \prod_{j=1}^m \bigl(\bar{P}_j Q P_j^\top\bigr) 
       \ge c \prod_{k=0}^n \lambda_{k}^{\varphi_Z(k)-\varphi_Z(k-1)} \,.
  \]
\end{Lemma}
\begin{proof}
  Let $B = b_0, \ldots, b_n$ be a unitary basis of~$\C^{n+1}$.
  Let $E_k = \langle b_0, \ldots, b_k \rangle$ the subspace generated
  by the first $k+1$ basis vectors. By definition of~$\varphi_Z$, the
  set $\Sigma(B) \subset S_m$ of permutations~$\sigma$ with the following
  property is nonempty:
  \[ P_{\sigma(j)} \notin E_k \quad\text{if}\quad j > \varphi_Z(k) \,. \]
  Define $k_\sigma(j) = \min\{k : \sigma(j) \le \varphi_Z(k)\}$; then
  $P_{\sigma(j)} \notin E_{k_\sigma(j)-1}$ if $\sigma \in \Sigma(B)$. Write
  $P_j = \sum_{i=0}^n \xi_{ji} b_i$ and define 
  \[ f_{\sigma}(B)
       = \prod_{j=1}^m \Bigl(\sum_{i=k_\sigma(j)}^n |\xi_{\sigma(j),i}|^2\Bigr)
       = \prod_{j=1}^m \Bigl(\sum_{i=k_\sigma(j)}^n
                             |\langle P_{\sigma(j)}, b_i \rangle|^2\Bigr)
  \]
  and
  \[ f(B) = \max\{f_{\sigma}(B) : \sigma \in S_m\} \,. \]
  It is clear that $f_\sigma$ is continuous on the set of unitary bases
  and that $f_\sigma(B) > 0$ if $\sigma \in \Sigma(B)$. This implies
  that $f$ is continuous and positive. Since
  the set of all unitary bases (i.e., $\U(n+1)$) is compact, there is
  some $c > 0$ such that $f(B) \ge c$ for all~$B$.
  
  Now let $Q$ be a positive definite Hermitian matrix as in the statement
  of the Lemma. Let $B = b_0, \dots, b_n$ be a unitary basis of eigenvectors
  such that $b_j Q = \lambda_j b_j$. We then have for $\sigma \in S_m$
  and using notation introduced above
  \begin{align*}
    \prod_{j=1}^m \bigl(\bar{P}_j Q P_j^\top\bigr)
      &= \prod_{j=1}^m \bigl(\bar{P}_{\sigma(j)} Q P_{\sigma(j)}^\top\bigr)
       =  \prod_{j=1}^m \Bigl(\sum_{i=0}^n \lambda_i |\xi_{\sigma(j),i}|^2\Bigr) \\
      &\ge \prod_{j=1}^m \Bigl(\lambda_{k_\sigma(j)}
                                \sum_{i=k_\sigma(j)}^n |\xi_{\sigma(j),i}|^2\Bigr) \\
      &= f_\sigma(B) \prod_{j=1}^m \lambda_{k_\sigma(j)}
       = f_\sigma(B) \prod_{k=0}^n \lambda_{k}^{\varphi_Z(k)-\varphi_Z(k-1)} \,.
  \end{align*}
  Taking the maximum over all $\sigma \in S_m$ now shows that
  \[ \prod_{j=1}^m \bigl(\bar{P}_j Q P_j^\top\bigr)
       \ge f(B) \prod_{k=0}^n \lambda_{k}^{\varphi_Z(k)-\varphi_Z(k-1)}
       \ge c \prod_{k=0}^n \lambda_{k}^{\varphi_Z(k)-\varphi_Z(k-1)} \,.
  \]
\end{proof}


\section{The Covariant}

\begin{Definition}
  Let $\tilde{\CZ}_m$ etc.\ denote the set of point clusters of degree~$m$
  with a choice of coordinates for the points, up to scaling the coordinates
  of the points with factors whose product is~1. We will call
  $\tilde{Z} \in \tilde{\CZ}_m$ a {\em point cluster with scaling}.
  
  For $\lambda \in \C^\times$ and $\tilde{Z} \in \tilde{\CZ}_m$, we write
  $\lambda \tilde{Z}$ for the cluster with scaling that we obtain by
  scaling one of the points in~$\tilde{Z}$ by~$\lambda$. This defines an
  action of~$\C^\times$ on~$\tilde{\CZ}_m$ such that the quotient
  $\C^\times \backslash \tilde{\CZ}_m$ is $\CZ_m$. If $\tilde{Z} \in \tilde{\CZ_m}$,
  then we write $Z$ for the image of~$\tilde{Z}$ in~$\CZ_m$. 
\end{Definition}

\begin{Definition}
  For a point cluster with scaling $\tilde{Z} \in \tilde{\CZ}_m$,
  pick a representative $\sum_{j=1}^m P_j$ with row vectors~$P_j$.
  Then, for $Q \in \HH_{n,\C}$, represented by a Hermitian matrix, we
  define
  \[ D_{\tilde{Z}}(Q)
       = D(\tilde{Z}, Q)
       = \sum_{j=1}^m \log (\bar{P}_j Q P_j^\top) - \frac{m}{n+1} \log \det Q \,.
  \]
\end{Definition}

$D(\tilde{Z},Q)$ is clearly invariant under scaling of~$Q$, and it does
not depend on the choice of representative for~$\tilde{Z}$.
Note also that for $\gamma \in G$,
\[ D(\tilde{Z} \cdot \gamma, Q \cdot \gamma) = D(\tilde{Z}, Q) \,. \]
Furthermore, we have $D(\bar{\tilde{Z}}, \bar{Q}) = D(\tilde{Z}, Q)$
and $D(\lambda \tilde{Z}, Q) = \log |\lambda|^2 + D(\tilde{Z}, Q)$.

This function generalises the distance function used in Prop.~5.3
of~\cite{StollCremona}. We will now proceed to show that for stable
clusters, there is a unique form $Q \in \HH_{n,\C}$ that minimises
this distance.

To that end, we now identify $\HH_{n,\C}$ with the set of positive definite
Hermitian matrices of determinant~$1$. This is a real $n(n+2)$-dimensional
submanifold of the space of all complex $(n+1) \times (n+1)$-matrices.
$\SL(n+1, \C)$ acts transitively on this space,
and the tangent space~$T$ at the identity matrix~$I$ consists of the
Hermitian matrices of
trace zero. We say that a twice continuously differentiable function
on~$\HH_{n,\C}$ is {\em convex} if its second derivative is positive
semidefinite, and {\em strictly convex} if its second derivative is
positive definite. Then the usual conclusions on convex functions apply.

\begin{Lemma} \label{Lemma:DZProp}
  Let $\tilde{Z} \in \tilde{\CZ}_m$ be a point cluster with scaling.
  \begin{enumerate}
    \item The function $D_{\tilde{Z}}$ is convex.
    \item If $Z$ is non-split, then $D_{\tilde{Z}}$ is strictly convex.
    \item If $Z$ is semi-stable, then $D_{\tilde{Z}}$ is bounded from below.
    \item If $Z$ is stable, then the sets
          $\{Q \in \HH_{n,\C} : D_{\tilde{Z}}(Q) \le B\}$ are compact for all 
          $B \in \R$.
  \end{enumerate}
\end{Lemma}
\begin{proof}
  Since scaling~$\tilde{Z}$ only changes $D_{\tilde{Z}}$ by an additive
  constant, we can assume that $\tilde{Z} = P_1 + \ldots + P_m$ with
  row vectors~$P_j$ satisfying $\|P_j\|^2 = 1$.
  
  (1) Since $D_{\tilde{Z}}(Q \cdot \gamma) =  D_{\tilde{Z} \cdot \gamma^{-1}}(Q)$,
      we can assume
      that $Q = I$. We compute the second derivative at $\lambda = 0$ of 
      $\lambda \mapsto f(\lambda) = D_{\tilde{Z}}\bigl(\exp(\lambda B)\bigr)$,
      where $B$ is a Hermitian trace-zero matrix (i.e., $B \in T$). We have
      \begin{align*}
        D_{\tilde{Z}}\bigl(\exp(\lambda B)\bigr)
          &= \sum_j \log(1 + \bar{P}_j B P_j^\top \cdot \lambda
                    + \bar{P_j} B^2 P_j^\top \cdot \lambda^2/2 + \ldots \;) \\
          &= \sum_j \left(\bar{P}_j B P_j^\top \cdot \lambda
              + (\bar{P}_j B^2 P_j^\top - (\bar{P}_j B P_j^\top)^2) \cdot \lambda^2/2 
              + \ldots \; \right)
      \end{align*}
      The second derivative therefore is
      \[ \sum_j \left(\bar{P}_j B^2 P_j^\top - (\bar{P}_j B P_j^\top)^2\right)
           = \sum_j \left(\|P_j \bar{B}\|^2 \|P_j\|^2
                           - |\langle P_j \bar{B}, P_j \rangle|^2\right)
           \ge 0
      \]
      by the Cauchy-Schwarz inequality.
      This shows that the second derivative is positive semidefinite,
      whence the first claim.
  
  (2) The second derivative vanishes
      exactly when every $P_j$ is an eigenvector of~$B$ for all~$j$. 
      Since $B$ has trace zero, there are at least two distinct
      eigenvalues (unless $B$ is the zero matrix). The `non-split' condition
      therefore excludes this possibility, implying that the second
      derivative at~$I$ is positive definite. Since the condition is
      invariant under the action of~$\SL(n+1, \C)$, the second derivative
      is positive definite everywhere.
  
  (3) By Lemma~\ref{Lemma:LowerBound}, we find some $c > 0$ such that
      for $Q \in \HH_{n,\C}$ with eigenvalues 
      $\lambda_0 \le \ldots \le \lambda_n$, we have
      \[ \prod_{j=1}^m \bigl(\bar{P}_j Q P_j^\top\bigr)
           \ge c \prod_{k=0}^n \lambda_k^{\varphi_Z(k)-\varphi_Z(k-1)} \,.
      \]
      With $\varphi_Z(k) \le (k+1)\frac{m}{n+1}$, we obtain
      \begin{align*}
        D_{\tilde{Z}}(Q)
          &\ge \log c
                + \sum_{k=0}^n
                    \left(\varphi_Z(k)-\varphi_Z(k-1)\right) \log \lambda_k \\
          &= \log c + m\,\log \lambda_n
              - \sum_{k=1}^n \varphi_Z(k-1)
                              (\log \lambda_k - \log \lambda_{k-1}) \\
          &\ge \log c + m\,\log \lambda_n
                - \frac{m}{n+1} 
                    \sum_{k=1}^n k (\log \lambda_k - \log \lambda_{k-1}) \\
          &= \log c
                + \frac{m}{n+1} \sum_{k=0}^n \log \lambda_k \\
          &= \log c
      \end{align*}
      (recall that $\sum_k \log \lambda_k = \log \det Q = 0$).
  
  (4) We now use that $\varphi_Z(k) \le (k+1)\frac{m}{n+1} - \frac{1}{n+1}$
      for $0 \le k \le n-1$. The computation in the proof of~(3) above then
      yields
      \begin{align*}
        D_{\tilde{Z}}(Q)
          &\ge \log c + m\,\log \lambda_n
              - \sum_{k=1}^n \varphi_Z(k-1)
                              (\log \lambda_k - \log \lambda_{k-1}) \\
          &\ge \log c + m\,\log \lambda_n \\
          &\qquad {} - \frac{m}{n+1} 
                         \sum_{k=1}^n k (\log \lambda_k - \log \lambda_{k-1})
              + \frac{1}{n+1} \sum_{k=1}^n (\log \lambda_k - \log \lambda_{k-1}) \\
          &= \log c + \frac{1}{n+1} (\log \lambda_n - \log \lambda_0) \,.
      \end{align*}
      So $D_{\tilde{Z}}(Q) \le B$ implies that $\lambda_n/\lambda_0$ is bounded,
      but this implies that the subset of $Q \in \HH_{n,\C}$ satisfying
      $D_{\tilde{Z}}(Q) \le B$ is also bounded. Since it is obviously closed, it
      must be compact.
\end{proof}

\begin{Remark}
  Note that if $Z$ is not stable, then there are sets $\{Q : D_{\tilde{Z}}(Q) \le B\}$
  that are not compact. Indeed, there is a linear subspace~$L_0 \subset \C^{n+1}$ 
  of some dimension
  $0 < k+1 < n+1$ containing at least $(k+1)m/(n+1)$ points of~$Z$. Let $L_1$ be its
  orthogonal complement. Let $Q_\lambda$ be the Hermitian matrix with
  eigenvalue $\lambda^{-(n-k)}$ on~$L_0$ and eigenvalue $\lambda^{k+1}$
  on~$L_1$. Then we have for $\lambda \ge 1$ that
  \[ D_{\tilde{Z}}(Q_\lambda)
      \le \text{const.} + (k+1)\frac{m}{n+1}\,\log \lambda^{-(n-k)}
                        + (n-k)\frac{m}{n+1}\,\log \lambda^{k+1}
      = \text{const.} \,;
  \]
  but the set $\{Q_\lambda : \lambda \ge 1\}$ is not relatively compact.
  
  We also see that $D_{\tilde{Z}}$ is not bounded from below
  when $Z$ is not semi-stable, since using the corresponding strict
  inequality, we find with a similar argument that
  \[ D_{\tilde{Z}}(Q_\lambda) \le \text{const.} - \varepsilon \log \lambda \]
  for some $\varepsilon > 0$.
\end{Remark}

\begin{Corollary} \label{Cor:ztheta}
  If $\tilde{Z} \in \tilde{\CZ}_m^\stab$, then the function $D_{\tilde{Z}}$
  has a unique critical
  point $z(Z)$ on~$\HH_{n,\C}$, and at this point $D_{\tilde{Z}}$ achieves its
  global minimum~$\log \theta(\tilde{Z})$ (for some $\theta(\tilde{Z}) \in \R_{>0}$).
\end{Corollary}
\begin{Proof}
  By Lemma~\ref{Lemma:DZProp}, we know that $D_{\tilde{Z}}$ is strictly convex
  and that the set $\{Q \in \HH_{n,\C} : D_{\tilde{Z}}(Q) \le B\}$ is always compact.
  The first property 
  implies that every critical point must be a local minimum. By the second
  property, there exists a global minimum. If there were two distinct
  local minima, then on a path joining the two, there would have to
  be a local maximum, but then the second derivative would not be
  positive definite in this point, a contradiction. Hence there is
  a unique local minimum, which must then also be the global minimum
  and the unique critical point.
  
  Since $D_{\lambda \tilde{Z}} = \log |\lambda|^2 + D_{\tilde{Z}}$, the
  minimising point in~$\HH_{n,\C}$ does not depend on the scaling, so it
  only depends on~$Z$, and the notation $z(Z)$ is justified.
\end{Proof}

Note that we have $\theta(\lambda \tilde{Z}) = |\lambda|^2 \theta(\tilde{Z})$.

Corollary~\ref{Cor:ztheta} defines $z : \CZ_m^\stab \to \HH_{n,\C}$ and
$\theta : \tilde{\CZ}_m^\stab \to \R_{> 0}$. The latter extends to
\[ \theta : \tilde{\CZ}_m \To \R_{> 0} \]
with the definition
$\theta(\tilde{Z}) = \inf_{Q \in \HH_{n,\C}} \exp\bigl(D(\tilde{Z},Q)\bigr)$. 
By Lemma~\ref{Lemma:DZProp},~(3), we have $\theta(\tilde{Z}) > 0$ if
$\tilde{Z} \in \tilde{\CZ}_m^\sstab$, and by the preceding remark,
$\theta(\tilde{Z}) = 0$ if $\tilde{Z}$ is not semi-stable.

\begin{Corollary}
  The function $z : \CZ_m^\stab \to \HH_{n,\C}$ is $\SL(n+1, \C)$-equivariant.
  It also satisfies $z(\bar{Z}) = \overline{z(Z)}$.
  In particular, $z$ restricts to $z : \CZ_m^\stab(\R) \to \HH_{n,\R}$.
  
  The function $\theta : \tilde{\CZ}_m \to \R_{> 0}$ is invariant
  under $\SL(n+1, \C)$ and under complex conjugation.
\end{Corollary}
\begin{Proof}
  The first statement follows from the invariance of~$D$ (under the action
  of both $\SL(n+1, \C)$
  and complex conjugation) and the uniqueness of~$z(Z)$. The second
  statement follows from the invariance of~$D$.
\end{Proof}

In some cases the point $z(Z)$ is uniquely determined by symmetry considerations.
Namely if the point cluster~$Z \in \CZ_m^\stab$ is stabilised by a
subgroup of~$\SL(n+1, \C)$ that
fixes a unique point in~$\HH_{n,\C}$, then $z(Z)$ must be this point,
compare Lemma~3.1 in~\cite{StollCremona}. This facilitates the
numerical computation of~$z(Z)$, since it eliminates the need for
finding numerically the minimum of the distance function on~$\HH_{n,\C}$.

\begin{Example}
  Consider a sum~$Z$ of $n+2$ points in general position in~$\BP^n(\C)$.
  Then $Z$ is stable. Since $\PGL(n+1, \C)$
  acts transitively on $(n+2)$-tuples of points in general position,
  we can assume that the points in~$Z$ are the coordinate points together with
  the point $(1 : \ldots : 1)$. Let this specific cluster be~$Z_0$.
  The stabiliser of~$Z_0$ in~$\PGL(n+1)$ is
  isomorphic to the symmetric group~$S_{n+2}$; its preimage $\Gamma$ in~$\SL(n+1,\C)$
  acts irreducibly on~$\C^{n+1}$. By Schur's lemma, there is a unique
  (up to scaling) $\Gamma$-invariant positive definite Hermitian form.
  It can be checked that
  \[ Q_0(x_0, \dots, x_n)
      = \sum_{i=0}^n |x_i|^2 + \sum_{0 \le i < j \le n} |x_i-x_j|^2
      = (n+2) \sum_{i=0}^n |x_i|^2 - \Bigl| \sum_{i=0}^n x_i \Bigr|^2
  \]
  is invariant under~$\Gamma$, hence $z(Z_0) = Q_0$. In general, we just
  have to find a matrix $\gamma$ such that $Z_0 \cdot \gamma^{-\top} = Z$;
  then
  \[ z(Z) = z(Z_0 \cdot \gamma^{-\top}) = Q_0 \cdot \gamma^{-\top} \,. \]
  Note that
  $Z_0 \cdot \gamma^{-\top} = \sum_j P_{0,j} \gamma$ if $Z_0 = \sum_j P_{0,j}$ and
  we think of the $P_{0,j}$ as row vectors. So if $Z = \sum_j P_j$, then
  the rows of~$\gamma$ are coordinate vectors for the first $n+1$ points
  in~$Z$, scaled in such a way that their sum is a coordinate vector for
  the last point.
\end{Example}


\section{Reduction of Point Clusters}

We can now define when a point cluster is reduced.

\begin{Definition}
  Let $Z \in \CZ_m^\stab(\R)$. We say that $Z$ is {\em LLL-reduced},
  resp., {\em Minkowski-reduced}, if the positive definite real quadratic
  form corresponding to~$z(Z)$ is LLL-reduced, resp., Minkowski-reduced.
\end{Definition}

By definition, there is an essentially unique Minkowski-reduced representative
in the $\SL(n+1, \Z)$-orbit of a given point cluster $Z \in \CZ_m^\stab(\R)$.
On the other hand, for computational purposes, it is usually more convenient
to work with LLL-reduced representatives. In order to find an LLL-reduced
representative of $Z$'s orbit, we compute the covariant $Q = z(Z)$. Then we
use the LLL algorithm~\cite{LLL} to find $\gamma \in \SL(n+1, \Z)$ such that
$Q \cdot \gamma$ is LLL-reduced. Then $Z \cdot \gamma$ is an LLL-reduced
representative of the orbit of~$Z$.

\begin{Example}
  We can use our results to reduce pencils of quadrics in three variables
  whose generic member is smooth. These correspond to four points in general
  position in~$\BP^2$. We illustrate the method with a concrete example.
  Let
  \begin{align*}
    Q_1(x,y,z)
      &= 857211194051 x^2 - 10879213981695 x y - 1296007209476 x z \\
      & \qquad {} + 34518126244996 y^2 + 8224075847095 y z + 489854396055 z^2 \\
    Q_2(x,y,z)
      &= 2274418654562 x^2 - 28865567091425 x y - 3438665984061 x z \\
      & \qquad {} + 91586146842213 y^2 + 21820750429746 y z + 1299719350945 z^2
  \end{align*}
  be a pair of quadrics. We first determine a good basis of the pencil
  spanned by $Q_1$ and~$Q_2$ by reducing the binary cubic
  \[ \det (xM_1 + yM_2) = 27348 x^3 + 215720 x^2 y + 567184 x y^2 + 497080 y^3 \]
  with the approach described in~\cite{StollCremona}. Here $M_1$ and~$M_2$ are
  the matrices of second partial derivatives of $Q_1$ and~$Q_2$, respectively.
  This suggests the new basis
  \[ Q'_1 = -21 Q_1 + 8 Q_2\,, \qquad Q'_2 = -8 Q_1 + 3 Q_2 \]
  with already somewhat smaller coefficients; the new binary cubic is
  \[ -4 x^3 + 88 x^2 y + 112 x y^2 - 24 y^3 \,. \]
  Now we find the four points of intersection numerically. We obtain
  \begin{align*}
    P_1 &= (0.3038054131 + 0.0003625989 i : -0.0712511408 + 0.0000571409 i : 1) \\
    P_2 &= (0.3038054131 - 0.0003625989 i : -0.0712511408 - 0.0000571409 i : 1) \\
    P_3 &= (0.3038639670 + 0.0003672580 i : -0.0712419135 + 0.0000578751 i : 1) \\
    P_4 &= (0.3038639670 - 0.0003672580 i : -0.0712419135 - 0.0000578751 i : 1)
  \end{align*}
  and from this a matrix $\gamma \in \SL(3,\C)$ that brings these points in standard 
  position:
  \[ \gamma^{-1}
       = \begin{pmatrix}
           -13584.01 - 1762.69 i &   3186.66 + 407.04 i & -44719.72 - 5748.66 i \\
            8318.54 + 10882.75 i & -1945.84 - 2556.21 i & 27338.35 + 35854.08 i \\
            14176.55 + 2104.80 i &  -3324.73 - 486.76 i &  46662.58 + 6870.37 i
         \end{pmatrix}
  \]
  From this, we obtain a matrix representing $z(P_1+P_2+P_3+P_4)$ as
  {\small
  \[ \bar{\gamma} \begin{pmatrix} 3 & -1 & -1 \\ -1 & 3 & -1 \\ -1 & -1 & 3
                  \end{pmatrix} \gamma^\top
       = \begin{pmatrix}
             241474533625.0 & -1532325529959.9 & -182541212588.9 \\
           -1532325529959.9 &  9723681808257.5 & 1158352212636.4 \\
            -182541212588.9 &  1158352212636.4 &  137990925143.2
         \end{pmatrix}
  \]
  }%
  (For the actual computation, more precision is needed than indicated by
  the numbers above.)
  An LLL computation applied to this Gram matrix suggests the transformation given by
  \[ g = \begin{pmatrix}
            3780 & 19276 & -12561 \\
            -889 & -4515 &   2953 \\
           12463 & 63400 & -41405
         \end{pmatrix}
  \]
  and indeed, if we apply the corresponding substitution to $Q'_1$ and~$Q'_2$,
  we obtain the nice and small quadrics
  \[ 2 x^2 - x y + x z + 2 z^2 \qquad\text{and}\qquad {-2} x z + 3 y^2 - y z + 2 z^2 \,.
  \]
\end{Example}


\section{Reduction of Ternary Forms}

In this section, we apply the reduction theory of point clusters to ternary
forms. The idea is to associate to a ternary form, or rather, to the plane
curve it defines, a stable point cluster in a covariant way. This should be a
purely geometric construction working over any base field of characteristic zero.

We will only consider irreducible ternary forms~$F$ of degree~$d$. Assume
that the curve defined by~$F$ has $r$ nodes and no other singularities;
then its genus is
\[ g = \half (d-1)(d-2) - r \,, \]
and by
\cite[Exercise IV.4.6, p.~337]{Hartshorne}, the number of inflection points is
\[ 6(g-1) + 3d = 3d(d-2) - 6r \,. \]
We let $Z(F)$ be the sum of the inflection points, counted with multiplicity.
When is $Z(F)$ stable? The first condition
is that the multiplicity of any point must be less than
$d(d-2) - 2r$. Now the multiplicity is $2$ less than the order of
tangency of the inflectional tangent, so it is at most~$d-2$.
Hence the condition is satisfied if $d-2 < d(d-2) - 2r$, i.e., if
$0 < (d-1)(d-2)/2 - r = g$.
The second condition is that the multiplicities of points on a line
add up to less than $2d(d-2) - 4r$. Since there are at most $d$ points
on the curve on a line, this sum is at most $d(d-2)$
Hence the condition is satisfied if $r < d(d-2)/4$. 

In any case, if $F$ defines a nonsingular plane curve of positive genus,
then $Z(F)$ is stable, and we can set $z(F) = z(Z(F))$. We then define $F$
to be {\em reduced} if $z(F)$ is reduced (i.e., if $Z(F)$ is reduced).

\begin{Example}
  If $F$ is a nonsingular cubic, then it defines a smooth curve~$C$ of
  genus~1, with Jacobian elliptic curve~$E$. The 3-torsion subgroup $E[3]$
  acts on~$C$ by linear automorphisms of the ambient~$\BP^2$. The preimage
  of~$E[3]$ in~$\SL(3,\C)$ is a nonabelian group~$\Gamma$ of order~$27$ that acts
  irreducibly on~$\C^3$. Therefore there is a unique $Q \in \HH_{2,\C}$
  that is invariant under the action of~$E[3]$. This $Q$ is then~$z(F)$.
  If we know explicit matrices $M_T \in \SL(3,\C)$ for $T \in E[3]$ that
  give the action of~$E[3]$ on~$\BP^2$, then we can compute a representative
  of~$Q$ as a Hermitian matrix as
  \[ Q = \sum_{T \in E[3]} \overline{M_T}^\top M_T \,, \]
  compare~\cite[\S~6]{CFS}.
  
  We get the same result if we consider the cluster of inflection points
  on~$C$, since this cluster (which is a principal homogeneous space for
  the action of~$E[3]$) is invariant under the same group~$\Gamma$.
  Numerically, however, the method using the action of~$E[3]$ seems to be
  more stable. See~\cite[\S~6]{CFS} for some more discussion and details.
\end{Example}

In general, we have to find the inflection points numerically and then find the
minimum of~$D_{\tilde{Z}}$, also numerically. This can be done by a steepest
descent method. We will illustrate this by reducing a ternary quartic.

\begin{Example}
  Let
  \begin{align*}
    F(x,y,z) 
      &= 390908548757 x^4 - 1083699236751 x^3 y + 835578482044 x^3 z \\
      & \qquad {}
        + 1126610184312 x^2 y^2 - 1737329379412 x^2 y z + 669777678687 x^2 z^2 \\
      & \qquad {}
        - 520542386163 x y^3 + 1204081445939 x y^2 z - 928398396271 x y z^2 \\
      & \qquad {}
        + 238611653627 x z^3 + 90192376558 y^4 - 278168756247 y^3 z \\
      & \qquad {}
        + 321720059816 y^2 z^2 - 165373310794 y z^3 + 31877479532 z^4 \,.
  \end{align*}
  We compute the inflection points as the intersection points of $F = 0$ and
  $H = 0$, where $H$ is the Hessian of~$F$. This gives 24~coordinate vectors
  and defines the point cluster~$\tilde{Z}$. We then use a steepest descent
  method to find (an approximation to) $z(Z)$, represented by the matrix
  \[ \begin{pmatrix}
       367751.9942 & -254909.8720 &  196557.1210 \\
      -254909.8720 &  176692.9800 & -136245.3974 \\
       196557.1210 & -136245.3974 &  105056.8935
     \end{pmatrix} \,.
  \]
  LLL applied to this Gram matrix suggests the transformation
  \[ \begin{pmatrix}
        -7 &  23 &  -89 \\
       -34 & 118 & -443 \\
       -31 & 110 & -408
      \end{pmatrix} \,,
  \]
  which turns $F$ into
  \[ 3 x^4 - 3 x^3 y + 3 x^3 z + x^2 y^2 - 2 x^2 z^2 + x y^2 z - x y z^2
      - 2 x z^3 + 3 y^4 - 3 y^3 z + y^2 z^2 - 3 z^4 \,.
  \]
\end{Example}



\begin{thebibliography}{MM}

\frenchspacing
\renewcommand{\baselinestretch}{1}

\bibitem{CFS}
  {\sc J.E. Cremona, T.A. Fisher} and {\sc M. Stoll},
  {\it Minimisation and reduction of 2-, 3- and 4-coverings of elliptic curves},
  Preprint (2009). arXiv:0908.1741v1 [math.NT]

\bibitem{Hartshorne}
  {\sc R. Hartshorne}, {\it Algebraic Geometry}, Springer GTM~{\bf 52},
  corr. 3rd printing (1983).

\bibitem{LLL}
  {\sc A.K. Lenstra, H.W. Lenstra} and {\sc L. Lov\'asz},
  {\it Factoring polynomials with rational coefficients},
  Math. Ann. {\bf 261}, no. 4, 515--534 (1982).

\bibitem{Mumford}
  {\sc D. Mumford, J. Fogarty and F. Kirwan},
  {\it Geometric invariant theory}, 3rd enl. ed.,
  Erg. Math. Grenzgeb. 3. Folge, vol.~34. Berlin: Springer-Verlag (1993).

\bibitem{StollCremona}
  {\sc M. Stoll} and {\sc J.E. Cremona},
  {\it On the reduction theory of binary forms}, \\
  J. reine angew. Math. {\bf 565}, 79--99 (2003).

\end{thebibliography}
\end{document}